\documentclass[12pt]{amsart}
\usepackage{amsaddr}
\usepackage{amssymb,latexsym,amsmath,enumerate,nicefrac}
\usepackage[bookmarks=true]{hyperref}
\usepackage[all]{xy}
\usepackage{tikz}
\usetikzlibrary{calc}
\usepackage{graphics}
\usepackage{dsfont}
\usepackage{amsfonts}
\usepackage{mathtools}

\def\Za{B_a}
\def\Zpa{B_{-a}}

\allowdisplaybreaks

\newtheorem{theorem}{Theorem}[section]
\newtheorem{lemma}[theorem]{Lemma}

\newtheorem{proposition}[theorem]{Proposition}
\newtheorem*{corollary}{Corollary}
\theoremstyle{definition}

\theoremstyle{remark}

\makeatletter
\newcommand{\vast}{\bBigg@{3.35}}
\newcommand{\Vast}{\bBigg@{5}}
\makeatother

\makeatletter
\newcommand*\bigcdot{\mathpalette\bigcdot@{.6}}
\newcommand*\bigcdot@[2]{\mathbin{\vcenter{\hbox{\scalebox{#2}{$\m@th#1\bullet$}}}}}
\makeatother

\newcommand{\e}{\mathrm e}

\renewcommand{\(}{\left(} 
\renewcommand{\)}{\right)}

\begin{document}

\title{Estimates of Certain Exit Probabilities for $p$-Adic Brownian Bridges}
\author{David Weisbart}\address{Department of Mathematics\\University of California, Riverside\\Riverside, California, USA}\email{weisbart@math.ucr.edu} 

\maketitle


\begin{abstract} 
For each prime $p$, a diffusion constant together with a positive exponent specify a Vladimirov operator and an associated $p$-adic diffusion equation.  The fundamental solution of this pseudo-differential equation gives rise to a measure on the Skorokhod space of $p$-adic valued paths that is concentrated on the paths originating at the origin.  We calculate the first exit probabilities of paths from balls and estimate these probabilities for the brownian bridges.
\end{abstract}


\tableofcontents

\section{Introduction}\label{back}

The study of diffusion equations in the $p$-adic setting began with Vladimirov's introduction in \cite{Vlad88} of a pseudo-differential operator acting on certain complex valued functions with domain in the $p$-adic numbers that is, in many respects, an analog of the classical Laplacian.  He further investigated in \cite{Vlad90} the spectral properties of this operator, now known as the Vladimirov operator.  Both Taibleson in \cite{Taib} and Saloff-Coste in \cite{SC1} wrote even earlier about pseudo-differential operators in the context of local fields, and Saloff-Coste studied such operators in \cite{SC2} in the more general setting of local groups.  Kochubei gave in \cite{koch92} the fundamental solution to the $p$-adic analog of the diffusion equation, with the operator introduced by Vladimirov replacing the Laplace operator.  Albeverio and Karwowski further investigated diffusion in the $p$-adic setting in \cite{alb}, constructing a continuous time random walk on $\mathds Q_p$, computing its transition semigroup and infinitesimal generator, and showing among other things that the associated Dirichlet form is of jump type.  Varadarajan further explored diffusion in the non-Archimedean setting in \cite{var97}, discussing an analog to the diffusion equation where the functions have domains contained in $[0, \infty)\times \mathcal S$, where $\mathcal S$ is a finite dimensional vector space over a division ring which is finite dimensional over a local field of arbitrary characteristic.  In the current work, we calculate the exit probabilities of $p$-adic brownian paths from balls centered at the origin.  Theorem~\ref{6:thm:6} gives estimates for the exit probabilities of the $p$-adic brownian bridges and is the principle novelty of the paper.  It is significant at least in part for its potential application in estimating expected values of functionals on path space.  The intuitive nature of these estimates given the exact calculation of the exit probabilities for $p$-adic brownian motion belies the nontrivial technical challenge in obtaining these estimates.

The current work finds immediate application in an upcoming study of Adelic path measures and both diffusion and Schr\"{o}dinger type equations in the Adelic setting.  Diffusion in the $p$-adic setting has a multitude of applications.  For example, ultrametricity arises in the theory of complex systems and many references cited by \cite[Chapter~4]{KKZ} study an area where estimates on exit probabilities should be significant, for example, the works \cite{ave2, ave3, ave4, ave} of Avetisov, Bikulov, Kozyrev, and Osipov dealing with $p$-adic models for complex systems.  The estimates for the $p$-adic brownian bridges is particularly significant to the study of Feynman-Kac integrals in the $p$-adic setting since these integrals involve integration of functionals on path spaces with respect to measures on the brownian bridges.  Digernes, Varadarajan, and Varadhan in \cite{DVV} studied the spectral properties of the hamiltonians associated to a large class of quantum system using path integral methods and proved the convergence of finite dimensional real quantum systems to their continuum limits.   Albeverio, Gordon, and Khrennikov studied finite approximation of quantum systems in the setting of locally compact abelian groups in \cite{agk}, obtaining some of the results of \cite{DVV} in this very general setting.  In \cite{BDW}, Bakken, Digernes, and Weisbart improved upon some of the results of \cite{agk}, but in the restricted setting of configuration spaces that are local fields.  The results of \cite{BDW} relied on path integral methods, mirroring the approach in \cite{DVV} to obtain uniformity in the estimates.  Due to the importance of path integral methods in the study of quantum systems, we should expect application of the current work to the further study of quantum systems in the $p$-adic setting.    In an upcoming article, we apply the current work to the study of brownian motion in the adelic setting.

Fix $p$ to be a prime number.  Section~\ref{two} reviews the basic properties of the $p$-adic numbers, denoted henceforth by $\mathds Q_p$.  It also gives an overview of diffusion in the $p$-adic setting.  Take the stochastic process $X$ to be the $p$-adic brownian motion associated to a diffusion constant $\sigma$ and a Vladimirov operator with exponent $b$.  The sample paths for $X$ are in the probability space $(D([0, \infty)\colon \mathds Q_{p}), P)$, where $D([0, \infty)\colon \mathds Q_{p})$ is the Skorokhod space of paths valued in $\mathds Q_{p}$ and $P$ is a probability measure on $D([0, \infty)\colon \mathds Q_{p})$ that is concentrated on the paths originating at 0.  Section~\ref{three} presents the main equalities of the paper, exact calculations of the exit probabilities of a $p$-adic brownian motion from balls of fixed radii containing 0.  Finally, Section~\ref{four} provides estimates of the first exit probabilities from balls for the brownian bridges conditioned to originate at a fixed point in the given ball and to be at a fixed point in the same ball at a later point in time. 
 

\section{Brownian Motion in the $p$-Adic Setting}\label{two}

Gouv\^{e}a's book \cite{Gov} provides an accessible introduction to $p$-adic analysis. For a further study of $p$-adic numbers and analysis in the $p$-adic setting, see the book \cite{Ram} of Ramakrishnan and Valenza, and the earlier work \cite{Weil} of Weil.  A standard reference in the field of $p$-adic mathematical physics, the book \cite{vvz} of Vladimirov, Volovich, and Zelenov offers an excellent introduction to $p$-adic analysis and both quantum theory and diffusion in the $p$-adic setting as does Kochubei's book \cite{koch01}, which provides a greater focus on $p$-adic diffusion.  Much of the current section is taken from \cite{bw} with only minor modification and is included here for the reader's convenience and to improve some of the prior exposition.  Billingsley's book \cite{bil2} is a standard reference in probability theory that supplements the relevant theory that this section summarizes.

\subsection{Basic Facts about $\mathds Q_p$}\label{basic1}

Let $p$ be a fixed prime number.  Denote by $|\cdot|$ the absolute value on $\mathds Q_p$ and denote respectively by $B_k(x)$ and $S_k(x)$ the ball and the circle of radii $p^k$, the compact open sets \[{B}_k(x) = \{y\in \mathds Q_{p}\colon |y-x| \leq p^k\} \quad {\rm and}\quad {S}_k(x) = \{y\in \mathds Q_{p}\colon |y-x| = p^k\}.\]  Denote by $\mathds Z_{p}$ the \emph{ring of integers}, the unit ball in $\mathds Q_{p}$. Note that for any integer $a$, the set ${p^{-a}\mathds Z_p}$ is the ball $B_a(0)$.  The field $\mathds  Q_{p}$ is a locally compact abelian group under the usual addition and is a totally disconnected topological space.  Let $\mu$ be the Haar measure on the additive group $\mathds Q_{p}$, normalized to be 1 on $\mathds Z_{p}$.  Define for each $x$ in $\mathds Q_{p}$ the unique function \begin{align}\label{theafunction}a_{x}\colon \mathds Z\to \{0,1,\dots, p-1\}\end{align} that has the property that \[x = \sum_{k\in\mathds Z}a_x(k)p^k.\]  Note that there is a natural number $N$ with the property that $k$ is larger than $N$ implies that $a_{x}(-k)$ is 0.  Define by $\{x\}$ the rational number \[\{x\} = \sum_{k<0} a_{x}(k)p^k,\] where the given sum is necessarily a finite sum.  With $\mathds S^1$ denoting the multiplicative group of unit complex numbers, denote by $\chi$ the rank 0 character \[\chi\colon \mathds Q_{p} \to \mathds S^1 \quad {\rm by} \quad \chi(x) = e^{2\pi{\sqrt{-1}}\{x\}}.\] The field $\mathds Q_{p}$ is self dual, meaning that for any character  $\phi$ on $\mathds Q_{p}$, there is an $\alpha$ in $\mathds Q_{p}$ so that for all $x$ in $\mathds Q_{p}$, \[\phi(x) = \chi(\alpha x).\]  The Fourier transform $\mathcal F$ on $L^2(\mathds Q_{p})$ is the unitary extension to all of $L^2(\mathds Q_{p})$ of the operator initially defined for each $f$ in $L^1(\mathds Q_{p})\cap L^2(\mathds Q_{p})$ by \[(\mathcal Ff)(x) = \int_{\mathds Q_{p}}\chi(-xy)f(y)\,{\rm d}\mu\!\(y\).\]   The inverse Fourier transform $\mathcal F^{-1}$ acts on any $f$ in $L^1(\mathds Q_{p})\cap L^2(\mathds Q_{p})$ by \[(\mathcal F^{-1}f)(y) = \int_{\mathds Q_{p}}\chi(xy)f(x)\,{\rm d}\mu\!\(x\).\] Denote by $SB(\mathds Q_{p})$ the \emph{Schwartz-Bruhat} space of complex valued, compactly supported, locally constant functions on $\mathds Q_{p}$.  This set of functions is the $p$-adic analog of the set of complex valued, compactly supported, smooth functions on $\mathds R$ with the important difference that, unlike in the real case, $SB\!\(\mathds Q_{p}\)$ is invariant under the Fourier transform. 

\subsection{Path Spaces} 

Let $\mathcal S$ be a Polish space.  A time interval $I$ is either the interval $[0,\infty)$ or for some positive $T$ the interval $[0,T]$.  Denote by $F(I\colon \mathcal S)$ the set of all functions mapping $I$ to $\mathcal S$. Refer to this space as the space of all paths on $\mathcal S$ with domain $I$.  A set $\mathcal P$ is a \emph{path space} if there is a time interval $I$ such that $\mathcal P$ is a subset of $F(I\colon \mathcal S)$.  An \emph{epoch} is a strictly increasing finite sequence valued in $I\setminus\{0\}$.  A \emph{route} is a finite sequence of Borel subsets of $\mathcal S$.  If $k$ is a natural number, $e$ is an epoch, and $U$ is a route with \[e = (t_1, \dots, t_k) \quad {\rm and}\quad U = (U_0, \dots, U_k),\] then a \emph{history} that pairs $e$ with $U$ is a sequence of ordered pairs with \[h = \(\(0, U_0\), \(t_1, U_1\), \dots, \(t_k, U_k\)\).\]   The length of a finite sequence is the cardinality of its domain, and so a history is a sequence given by the term-wise pairing of an epoch with a route where the length of the route is equal to the length of the epoch plus 1. The additional initial place in the route is an initial location.  If $h$ is a history, then denote by $e(h)$ the epoch associated to $h$, by $U(h)$ the route associated to $h$, and by $\ell(h)$ the length of $e(h)$.  Denote by $H(I\colon \mathcal S)$ the set of all histories of $\mathcal S$ with time interval $I$ and define a function $C$ that associates to each $h$ in $H(I\colon \mathcal S)$ a set $C(h)$ by \begin{align*}C(h) &= \left\{\omega \in F(I\colon \mathcal S)\colon \omega(0) \in U_h(0) \right.\\&\qquad\qquad\left.\; {\rm and}\; i\in \mathds N\cap (0, {\rm length}(h)] \implies \omega\(e(h)_i\) \in U_h(i)\right\}.\end{align*}  Refer to the elements of $C(H(I\colon \mathcal S))$ as \emph{simple cylinder sets}.

For each $t$ in $I$, suppose that $\mathcal S_t$ is equal to $\mathcal S$.  For any finite sequence $(t_1, \dots, t_k)$ of points in $I$, denote by $\pi_{t_1, \dots, t_k}$ the canonical projection \[\pi_{t_1, \dots, t_k}\colon \prod_{t\in I}\mathcal S_t\to \mathcal S^k.\]  These projections are the \emph{length $k$ projections}.  The \emph{set of paths with $k$ conditions} is the set of all inverse images of Borel subsets of $\mathcal S^k$ under length $k$ projections.  The algebra of \emph{cylinder sets of finite type} is the union over $k$ in the natural numbers of the sets of paths with $k$ conditions and the set of \emph{cylinder sets} is the $\sigma$-algebra generated by the cylinder sets of finite type. The simple cylinder sets do not form an algebra, but they do form a $\pi$-system that generates the cylinder sets.

If a premeasure on the set of simple cylinder sets with paths valued in a Polish space and with epochs valued in a time interval $I$ is consistent, then the Kolmogorov Extension Theorem permits an extension of the premeasure to a measure on the space of all paths, denoted $F(I\colon \mathcal S)$.  Analytical investigations of path spaces often require specialization to more a restrictive path space.  Denote by $D(I \colon \mathcal S)$ the set of Skorohod paths, the set of c\'{a}dl\'{a}g functions from $I$ to $\mathcal S$ equipped with the Skorohod metric.  Measures on this space will be the central objects of our study.  If $P$ is a measure on $D(I \colon \mathcal S)$, then define for each $t$ in $I$ the random variable $X_t$ on the probability space $(D(I \colon \mathcal S), P)$ by \[X_t(\omega) = \omega(t)\quad \text{with}\quad \omega \in D(I \colon \mathcal S).\] The stochastic process $X$ is a function defined by \[X\colon t\mapsto X_t.\]  Suppose that $(A_i)_{i\in \{1, \dots, n\}}$ is a finite sequence of  subsets of $\mathcal S$ and $A$ is a subset of $\mathcal S$, that $(t_1, \dots, t_n)$ is a finite sequence in $I$ and $t$ is in $I$, and that $a$ is an element of $\mathcal S$.  The paper uses the following probabilist's convention for denoting sets, namely \begin{align*}&(X_{t}\in A) = \{\omega\in D(I \colon \mathcal S)\colon \omega(t)\in A\}, \quad (X_t = a) = (X_t \in \{a\}),\end{align*} and more generally, \[(X_{t_1}\in A_1\cap \dots \cap X_{t_n}\in A_n) = \{\omega\in D(I \colon \mathcal S)\colon \forall i\in\{1, \dots, n\}, \omega(t_i)\in A_i\}.\]

\subsection{Diffusion in the $\mathds Q_p$ Setting}\label{basic2}

Fix a positive real number $b$.  Define the multiplication operator ${\mathcal M}$ to act on $SB(\mathds Q_{p})$ by \[({\mathcal M}f)(x) = |x|^bf(x).\]  The pseudo Laplace operator $\tilde{\Delta}$ with exponent $b$ acts on $SB(\mathds Q_{p})$ by \begin{equation}\label{sec1:pseudoL}\big(\tilde{\Delta} f\big)(x) = \big(\mathcal F^{-1}{\mathcal M}\mathcal Ff\big)\!(x).\end{equation} This operator is densely defined and the Fourier transform of a multiplication operator, hence an essentially self adjoint operator from $SB(\mathds Q_p)$ to $L^2(\mathds Q_p)$.  Denote once again by $\tilde{\Delta}$ the unique self adjoint extension of the operator given by \eqref{sec1:pseudoL} to its maximal domain in $L^2(\mathds Q_{p})$.  Extend $\tilde{\Delta}$ to act on any complex valued function \[f\colon \mathds R_+\times \mathds Q_{p}\to \mathds C\] with the property that $f(t, \cdot)$ is in $\mathcal D(\tilde{\Delta})$ for each positive $t$ by fixing $t$ and viewing $f(t,\cdot)$ as a function only of $\mathds Q_{p}$.  Denote this extension by $\Delta$, the \emph{Vladimirov operator with exponent} $b$.  Define similarly the Fourier and inverse Fourier transforms on functions on $\mathds R_+\times \mathds Q_{p}$ that for each positive $t$ are square integrable over $\mathds Q_{p}$ by computing the given transform of the function for fixed positive $t$.  Fix $\sigma$ to be a positive real number and refer to it as a diffusion constant.  The pseudo differential equation \begin{align}\label{eq1} \dfrac{{\rm d}f(t,x)}{{\rm d}t} = -\sigma\Delta f(t,x)\end{align} has as its fundamental solution the function \[\rho(t,x) = \left(\mathcal F^{-1}e^{-\sigma t|\cdot|^b}\right)\!(x).\]  %
A minor modification of the more general arguments of \cite{var97} show that if $I$ is any time interval, then the function $f(t,x)$ is a probability density function that gives rise to a probability measure $P$ on $D(I\colon \mathds Q_{p})$ that is concentrated on the set of paths originating at 0.  In particular, if $h$ is a history and $U(h)_0$ is the set $\{0\}$, then define $P(C(h))$ by \begin{align}\label{meas}P(C(h)) &= \int_{U(h)_1} \cdots \int_{U(h)_{\ell(h)}} \rho(e(h)_1, x_1) \rho(e(h)_2 - e(h)_1, x_2 - x_1)\\&\hspace{.5in}\cdots \rho(t_{\ell(h)} - e(h)_{\ell(h)-1}, x_{\ell(h)} - x_{\ell(h)-1}) \,{\rm d}\mu\!\(x_{\ell(h)}\)\cdots \,{\rm d}\mu\!\(x_1\).\notag\end{align}  If 0 is not in $U_h(0)$, then $P(C(h))$ is 0.  This premeasure on the simple cylinder sets extends to the measure on $D(I \colon \mathds Q_{p})$.  Denote once again and henceforth by $P$ this probability measure.  Define the stochastic process $X$ to be the function \begin{align}\label{pProc}X \colon I\times D(I\colon \mathds Q_{p}) \to \mathds Q_{p} \quad {\rm by}\quad (t, \omega) \mapsto X_t(\omega) = \omega(t).\end{align}

For each time interval $I$, the probability measure $P$ on $D(I \colon \mathds Q_{p})$ gives full measure to paths originating at 0.   For each $x$ in $\mathds Q_{p}$, define $P_{x}$ to be the probability measure given by the same density function that defines $P$ but conditioned to give full measure to the paths originating at $x$.  If $A$ is a cylinder set, then \begin{equation}\label{conditionalProb}P_{x}(A) = P(A - x).\end{equation} For each $y$ in $\mathds Q_{p}$ and each positive $T$ in $I$, the arguments of \cite{var97} guarantee the existence of the probability measures concentrated on the $p$-adic brownian bridges, namely, the measures $P_{T, x, y}$ which are given by the measures $P_{x}$ conditioned so that paths almost surely take value $y$ at time $T$.  These conditioned measures form a continuous family of probability measures depending on the starting and ending points, as \cite{var97} discusses in its more general setting but with the diffusion constant $\sigma$ restricted to be equal to 1.  There is no obstruction to allowing for a more general diffusion constant, which \cite{bw} discusses.


\section{First Exit Probabilities for $p$-Adic Path Spaces}\label{three}

This section presents an exact calculation of the probability that a $p$-adic brownian path remains within a ball until time $T$.  Such an event is the complement of a first exit event occurring within the specified time interval.  This calculation allows for estimates of the same quantity in the context of the brownian bridges.  Denote by $\alpha$ the quantity \begin{align}\label{alpha}\alpha = 1 - \dfrac{p^b-1}{p^{b+1}-1}.\end{align}

Given a measurable subset $A$ of $\mathds Q_p$ and a function $f$ that is integrable over $A$, compress notation by writing \[\int_Af(x)\,{\rm d}x = \int_Af(x)\,{\rm d}\mu(x)\] and retain this notation throughout this and the next section. Let $X$ be the stochastic process given by \eqref{pProc}.  Define by $||X||_T$ the value \[||X||_T=\sup_{0\leq t\leq T}|X_t|.\]

\begin{theorem}\label{MainEstimate}
For any non-negative real number $T$ and integer $a$, \[P\left(||X||_T\leq p^a\right) = {\rm e}^{-\sigma\alpha Tp^{-ab}}.\]
\end{theorem}   
  
\begin{proof}
The non-Archimedean property of the absolute value on $\mathds Q_p$ guarantees the equality of the sets $A_1$ and $A_2$ where \[A_1 = \left\{\left(x_1, x_2, \dots, x_n\right)\colon \left|x_1\right| \leq p^a, \left|x_2\right| \leq p^a, \dots, \left|x_n\right| \leq p^a\right\}\] and \[A_2 = \left\{\left(x_1, x_2, \dots, x_n\right)\colon \left|x_1\right| \leq p^a, \left|x_2-x_1\right| \leq p^a, \dots, \left|x_n-x_{n-1}\right| \leq p^a\right\}.\]   For any epoch $(t_1, \dots, t_n)$ in $[0,T]$, the equality of $A_1$ and $A_2$ implies that \begin{align}\label{sec3:maxtoprodequality}&P\Big(\max_{t_j} \big(\big|X_{t_1}\big|, \dots, \big|X_{t_n}\big|\big) \leq p^a\Big)\notag\\&\hspace{.5in} = P\Big(\max_{t_j} \big(\big|X_{t_1}\big|, \big|X_{t_1}- X_{t_2}\big|, \dots, \big|X_{t_n}- X_{t_{n-1}}\big|\big) \leq p^a\Big).\end{align}  For each natural number $j$ in $[1, N]$, denote by $t_j$ the number \[t_j = \frac{jT}{N}.\]  Define the random variable $X_0$ to be  identically 0.  For each $j$, the increments $X_{t_j} - X_{t_{j-1}}$ are independent and identically distributed, implying that \begin{align}\label{sec3:prodeqa1}P\Big(\max_{t_j} \big(\big|X_{t_1}\big|, \dots, \big|X_{t_n}\big|\big) \leq p^a\Big)  &= \prod_{1\leq j\leq N}P\big(\big|X_{t_1}\big| \leq p^a\cap \dots\cap \big|X_{t_n}\big| \leq p^a\big)\notag\\& = P\big(\big|X_{t_1}\big| \leq p^a\big)^N = P\big(\big|X_{\frac{T}{N}}\big| \leq p^a\big)^N.\end{align}
Define the function $B$ on the non-negative real numbers by \[B(t) = P\big(|X_{t}| \leq p^a\big).\]  The twice continuous differentiability in $t$ of $B$ implies that \[B\Big(\frac{T}{n}\Big) = 1 + \frac{B^\prime(0)T}{n} + O\Big(\frac{1}{n^2}\Big),\] and so \begin{align*}\lim_{n\to \infty} B\Big(\frac{T}{n}\Big)^n = {\rm e}^{TB^\prime(0)}.\end{align*} The right continuity of the sample paths together with \eqref{sec3:maxtoprodequality} and \eqref{sec3:prodeqa1} guarantees that
\begin{align}\label{sec3:expforB}P\Big(||X||_T\leq p^a\Big) &= \lim_{n\to \infty}P\Big(\max_{j} \big(\big|X_{\frac{T}{n}}\big|, \big|X_{\frac{2T}{n}}\big|, \dots, \big|X_{\frac{jT}{n}}\big|, \dots, \big|X_{T}\big|\big) \leq p^a\Big)\notag\\ &= \lim_{n\to \infty} B\Big(\frac{T}{n}\Big)^n\notag\\ &= {\rm e}^{TB^\prime(0)}.\end{align}

Denote respectively by $B_r$ and $S_r$ the sets $B_r(0)$ and $S_r(0)$ and by ${\mathds 1}_{r}$ the characteristic function on $B_r$.  A straightforward modification of the arguments in \cite{var97} shows that, for positive $t$, the density function $\rho(t,x)$ for the random variable $X_t$ satisfies the equality
\begin{align}\label{sec3:formulaforf}\rho(t,x) &= \sum_{r\in \mathds Z} \e^{-\sigma tp^rb}\int_{S_r} \mathds \chi(xy)\,{\rm d}y\notag\\ &= \sum_{r\in \mathds Z} \Big(\e^{-\sigma tp^{rb}} - \e^{-\sigma tp^{(r+1)b}}\Big)\int_{B_r} \mathds \chi(xy)\,{\rm d}y\notag\\& = \sum_{r\in\mathds Z}\Big(\e^{-\sigma t p^{rb}} - \e^{-\sigma  t p^{(r+1)b}}\Big)p^r{\mathds 1}_{-r}(x).\end{align}Use equation \eqref{sec3:formulaforf} for $\rho$ to obtain the equality \begin{align}\label{tbyt}B(t) &= P\big(|X_{t}| \leq p^a\big) \notag\\&= \int_{B_a} \rho(t,x)\,{\rm d}x\notag\\ &= \sum_{r\in \mathds Z} \Big(\e^{-\sigma tp^{rb}} - \e^{-\sigma tp^{(r+1)b}}\Big)p^r\int_{B_a} \mathds 1_{-r}(x)\,{\rm d}x\notag\\& = \e^{-\sigma tp^{-ab}} + \sum_{r\leq -a-1}p^{a+r}\Big(\e^{-\sigma tp^{rb}} - \e^{-\sigma tp^{(r+1)b}}\Big).\end{align} The partial sums of the infinite sum \eqref{tbyt} as well as the partial sums of the derivatives of the terms of \eqref{tbyt} converge uniformly to continuous functions of time, justifying a term by term differentiation of the infinite sum.  Term by term differentiation of \eqref{tbyt} implies that %
\begin{align}\label{sec3:Bprimeiszero}B^\prime(0) &= -\sigma p^{-ab} + \sigma \sum_{r\leq -a-1}p^{a+r}\Big(p^{(r+1)b} - p^{rb}\Big)\notag\\
& = -\sigma p^{-ab} + \sigma p^a\big(p^b - 1\big)\sum_{r\leq -a-1}p^{r(1+b)}\notag\\
&=-\sigma p^{-ab} + \sigma p^a\big(p^b - 1\big)p^{-(a+1)(1+b)}\cdot\dfrac{1}{1-p^{-(1+b)}}\notag\\
&= -\sigma p^{-ab}\Big(1 - \frac{p^b-1}{p^{1+b} - 1}\Big) = -\sigma \alpha p^{-ab}.\end{align} Equalities \eqref{sec3:expforB} and \eqref{sec3:Bprimeiszero} together imply that \[P\big(||X||_T\leq p^a\big) = \e^{-\sigma \alpha Tp^{-ab}}.\]
\end{proof}

A change of variables together with \eqref{conditionalProb} implies the following further result that for every non-negative real number $T$ and $x$ in $\mathds Q_p$, \begin{equation}\label{sec3:precor:eq}P_x\!\left(||X-x||_T\leq p^a\right) = {\rm e}^{-\sigma\alpha Tp^{-ab}}.\end{equation} Every point of a ball in $\mathds Q_p$ is the center of the ball, which implies the following corollary to Theorem~\ref{MainEstimate}.

\begin{corollary}
For every non-negative real number $T$, for any $x$ in $\mathds Q_p$ and for any $x^\prime$ in $B_a(x)$, \[P_x\!\left(||X-x^\prime||_T\leq p^a\right) = {\rm e}^{-\sigma\alpha Tp^{-ab}}.\]
\end{corollary}



\section{First Exit Probabilities for the Brownian Bridges}\label{four}

Suppose that $I$ is a time interval containing the positive real number $t$.  Denote by $D_{t, x, y}(I\colon \mathds Q_p)$ the set of all paths in the $\mathds Q_p$ valued Skorokhod space on $I$ that start at $x$ at time 0 and are at $y$ at time $t$, the set of $p$-adic brownian bridges with the given endpoints.  The measure $P_{t, x, y}$ is the measure $P_x$ conditioned so that $D_{t, x, y}(I\colon \mathds Q_p)$ has full measure.    Let $r$ henceforth denote an index variable that varies over the set of integers and let $a$ be an integer.  The probabilities associated to first exit times for $p$-adic brownian paths imply a similar inequality for the brownian bridges, in particular, that if $y$ is in $p^{-a}\mathds Z_p$ and $t$ is in $(0,T]$, then \[P_{t, x, y}\big(||X-x||_T\leq p^a\big) \ge P_x(||X-x||_T\leq p^a).\] The proof requires several calculations that the lemmata below present.  To compress notation, recall that $p^{-a}\mathds Z_p$ is the set $B_a$ and use the latter notation for a ball henceforth.

\begin{lemma}\label{6:lem:1} Suppose that $\rho$ is the fundamental solution to \eqref{eq1}.  For all positive $t$, 
\[\int_{{\Za}} \rho(t,x)\,{\rm d}x = p^a\int_{{\Zpa}} \e^{-\sigma t |x|^b}\,{\rm d}x.\]
\end{lemma}

\begin{proof} Integrate $\rho(t,\cdot)$ over $\Za$ using the simplified expression for $\rho$ given by \eqref{sec3:formulaforf} to obtain
\begin{align*}
\int_{{\Za}} \rho(t,x)\,{\rm d}x & =  \sum_{r\in\mathds Z}\Big(\e^{-\sigma t p^{rb}} - \e^{-\sigma t p^{(r+1)b}}\Big)p^r\int_{{\Za}}{\mathds 1}_{-r}(x)\,{\rm d}x%
\\& = \sum_{r\leq -a}\Big(\e^{-\sigma t p^{rb}} - \e^{-\sigma t p^{(r+1)b}}\Big)p^rp^a + \sum_{r > -a}\Big(\e^{-\sigma t p^{rb}} - \e^{-\sigma t p^{(r+1)b}}\Big)%
\\& = \e^{-\sigma t p^{-ab}} + \Big(\e^{-\sigma t p^{-(a+1)b}} - \e^{-\sigma t p^{-ab}}\Big)p^{-1} \\&\hspace{2.35in}+  \Big(\e^{-\sigma t p^{-(a+2)b}} - \e^{-\sigma t p^{-(a+1)b}}\Big)p^{-2}+ \cdots%
\\& = \e^{-\sigma t p^{-ab}}\big(1-p^{-1}\big) + \e^{-\sigma t p^{-(a+1)b}}p^{-1}\big(1-p^{-1}\big)  \\&\hspace{2.35in}+ \e^{-\sigma t p^{-(a+2)b}}p^{-2}\big(1-p^{-1}\big) + \cdots%
\\& = p^a\left(\int_{S_{-a}} \e^{-\sigma t |x|^b}\,{\rm d}x + \int_{S_{-a-1}} \e^{-\sigma t |x|^b}\,{\rm d}x + \int_{S_{-a-2}} \e^{-\sigma t |x|^b}\,{\rm d}x +\cdots\right)\\& = p^a\int_{{\Zpa}} \e^{-\sigma t |x|^b}\,{\rm d}x.%
\end{align*}
\end{proof}

\begin{lemma}\label{6:lem:2}
Suppose that $t$ and $t^\prime$ are both positive real numbers. For all $z$ in $\mathds Q_p$, 
\begin{align*}&\int_{\mathds Q_p}\chi(zy)\e^{-\sigma t^\prime |y|^b}\bigg(\int_{{\Za}}\e^{-\sigma t|y+w|^b}\,{\rm d}w\bigg)\,{\rm d}y  \\&\hspace{2.1in}> \int_{\mathds Q_p}\chi(zy)\e^{-\sigma(t+t^\prime) |y|^b}{\rm d}y \int_{{\Za}}\e^{-\sigma t|w|^b}\,{\rm d}w.\end{align*}
\end{lemma}

\begin{proof}
Decompose the integral over $\mathds Q_p$ into a sum of integrals over disjoint circles to obtain the equalities
\begin{align}\label{6:Eqlem:2}
&\int_{\mathds Q_p}\chi(zy)\e^{-\sigma t^\prime |y|^b}\bigg(\int_{{\Za}}\e^{-\sigma t|y+w|^b}\,{\rm d}w\bigg)\,{\rm d}y \notag\\&\qquad= \sum_{r\in \mathds Z}\e^{-\sigma t^\prime p^{rb}}\int_{S_r}\chi(zy)\bigg(\int_{{\Za}}\e^{-\sigma t|y+w|^b}\,{\rm d}w\bigg)\,{\rm d}y%
\notag\\&\qquad =\sum_{r> a}\e^{-\sigma t^\prime p^{rb}}\int_{S_r}\chi(zy)\bigg(\int_{{\Za}}\e^{-\sigma t|y+w|^b}\,{\rm d}w\bigg)\,{\rm d}y%
\notag\\&\qquad\qquad +\sum_{r\leq a}\e^{-\sigma t^\prime p^{rb}}\int_{S_r}\chi(zy)\bigg(\int_{{\Za}}\e^{-\sigma t|y+w|^b}\,{\rm d}w\bigg)\,{\rm d}y%
\notag\\&\qquad =\sum_{r>a}\e^{-\sigma t^\prime p^{rb}}\int_{S_r}\chi(zy) p^a\e^{-\sigma t p^{rb}}\,{\rm d}y \notag\\&\hspace{2in}+\sum_{r\leq a}\e^{-\sigma t^\prime p^{rb}}\int_{S_r}\chi(zy)\,{\rm d}y\int_{{\Za}}\e^{-\sigma t|w|^b}\,{\rm d}w%
\notag\\&\qquad =\sum_{r>a}\e^{-(t + t^\prime)\sigma p^{rb}}\int_{S_r}\chi(zy)\,{\rm d}y \notag\\&\hspace{2in}+\sum_{r\leq a}\e^{-\sigma t^\prime p^{rb}}\int_{{\Za}}\e^{-\sigma t|w|^b}\,{\rm d}w\int_{S_r}\chi(zy)\,{\rm d}y%
\notag\\&\qquad =p^a\sum_{r>a}\e^{-(t + t^\prime)\sigma p^{rb}}\bigg(\int_{B_r}\chi(zy)\,{\rm d}y - \int_{B_{r-1}}\chi(zy)\,{\rm d}y\bigg)%
\notag\\&\hspace{1in} +\sum_{r\leq a}\e^{-\sigma t^\prime p^{rb}}\int_{{\Za}}\e^{-\sigma t|w|^b}\,{\rm d}w\bigg(\int_{B_r}\chi(zy)\,{\rm d}y - \int_{B_{r-1}}\chi(zy)\,{\rm d}y\bigg) %
\notag\\&\qquad =p^a\sum_{r>a}\Big(\e^{-\sigma(t + t^\prime) p^{rb}}-\e^{-\sigma(t + t^\prime) p^{(r+1)b}}\Big)\int_{B_r}\chi(zy)\,{\rm d}y - p^a\e^{-\sigma(t + t^\prime) p^{ab}}\int_{B_{a}}\chi(zy)\,{\rm d}y%
\notag\\&\hspace{1in}+\sum_{r\leq a}\bigg(\e^{-\sigma t^\prime p^{rb}} -\e^{-\sigma t^\prime p^{(r+1)b}}\bigg) \int_{B_r}\chi(zy)\,{\rm d}y\int_{{\Za}}\e^{-\sigma t|w|^b}\,{\rm d}w%
\notag\\&\hspace{2in} + \e^{-\sigma t^\prime p^{ab}}\int_{B_a}\chi(zy)\,{\rm d}y\int_{{\Za}}\e^{-\sigma t|w|^b}\,{\rm d}w = \ast.%
\end{align}
The estimate \[\e^{-\sigma t p^{ab}}< p^{-a}\int_{{\Za}}\e^{-\sigma t|w|^b}\,{\rm d}w < 1,\] and \eqref{6:Eqlem:2} together imply that
\begin{align*}
\ast & >\sum_{r>a}\Big(\e^{-\sigma(t + t^\prime) p^{rb}}-\e^{-\sigma(t + t^\prime) p^{(r+1)b}}\Big)\int_{B_r}\chi(zy)\,{\rm d}y\int_{{\Za}}\e^{-\sigma t|w|^b}\,{\rm d}w%
\\&\qquad\qquad +\sum_{r\leq a}\bigg(\e^{-\sigma t^\prime p^{rb}} -\e^{-\sigma t^\prime p^{(r+1)b}}\bigg) \int_{B_r}\chi(zy)\,{\rm d}y\int_{{\Za}}\e^{-\sigma t|w|^b}\,{\rm d}w.%
\end{align*}
For any $r$ in $\mathds Z$,%
\begin{align*}
g(s) = \int_{p^{rb}}^{p^{(r+1)b}} \e^{-sx}\,{\rm d}x = \e^{-sp^{rb}} - \e^{-sp^{(r+1)b}}
\end{align*}
is decreasing in the variable $s$ since 
\begin{align*}
g^\prime(s) = \int_{p^{rb}}^{p^{(r+1)b}} -x\e^{-sx}\,{\rm d}x <0.
\end{align*}
The fact that $g$ is decreasing in $s$ implies that        
\begin{align*}
\ast &> \sum_{r\in {\mathds Z}}\(\e^{-\sigma(t + t^\prime) p^{rb}}-\e^{-\sigma(t + t^\prime) p^{(r+1)b}}\)\int_{B_r}\chi(zy)\,{\rm d}y\int_{{\Za}}\e^{-\sigma t|w|^b}\,{\rm d}w
\\&\qquad = \sum_{r\in {\mathds Z}} \e^{-\sigma(t + t^\prime) p^{rb}}\int_{S_r}\chi(zy)\,{\rm d}y\int_{{\Za}}\e^{-\sigma t|w|^b}\,{\rm d}w
\\&\qquad = \sum_{r\in {\mathds Z}} \int_{S_r}\chi(zy)\e^{-\sigma(t + t^\prime) |y|^b}\,{\rm d}y \int_{{\Za}}\e^{-\sigma t|w|^b}\,{\rm d}w
\\&\qquad = \int_{\mathds Q_p}\chi(zy)\e^{-\sigma(t + t^\prime) |y|^b}\,{\rm d}y \int_{{\Za}}\e^{-\sigma t|w|^b}\,{\rm d}w.
\end{align*}

\end{proof}

\begin{proposition}\label{6:prop:3}
Suppose that $\rho$ is the fundamental solution to \eqref{eq1}.  For all positive real numbers $t$ and $t^\prime$ and all $z$ in $\mathds Q_p$, 
\[\int_{{\Za}}\rho(t,x)\rho(t^\prime,z-x)\,{\rm d}x > \rho(t+t^\prime, z)\int_{{\Za}}\rho(t,x)\,{\rm d}x.\]
\end{proposition}

\begin{proof}
Write both $\rho(t,x)$ and $\rho(t^\prime,z-x)$ as the Fourier transforms of exponential functions and rearrange terms to obtain the equalities
\begin{align}\label{sec3:intermediateineqtplustprime}
&\int_{{\Za}}\rho(t,x)\rho(t^\prime,z-x)\,{\rm d}x\notag%
\\&\qquad\qquad= \int_{{\Za}}\bigg(\int_{\mathds Q_p}\chi(xw)\e^{-\sigma t |w|^b}\,{\rm d}w\bigg)\bigg(\int_{\mathds Q_p}\chi((z-x)y)\e^{-\sigma t^\prime |y|^b}\,{\rm d}y\bigg)\,{\rm d}x\notag%
\\&\qquad\qquad= \int_{{\Za}}\bigg(\int_{\mathds Q_p}\int_{\mathds Q_p} \chi(xw)\chi((z-x)y)\e^{-\sigma t |w|^b}\e^{-\sigma t^\prime |y|^b}\,{\rm d}w\,{\rm d}y\bigg)\,{\rm d}x\notag%
\\&\qquad\qquad= \int_{\mathds Q_p}\int_{\mathds Q_p} \bigg(\int_{{\Za}} \chi(xw)\chi((z-x)y)\e^{-\sigma t |w|^b}\e^{-\sigma t^\prime |y|^b}\,{\rm d}x\bigg)\,{\rm d}w\,{\rm d}y\notag%
\\&\qquad\qquad= \int_{\mathds Q_p}\int_{\mathds Q_p} \bigg(\int_{{\Za}} \chi(xw)\chi(zy)\chi(-xy)\e^{-\sigma t |w|^b}\e^{-\sigma t^\prime |y|^b}\,{\rm d}x\bigg)\,{\rm d}w\,{\rm d}y\notag%
\\&\qquad\qquad= \int_{\mathds Q_p}\int_{\mathds Q_p} \bigg(\int_{{\Za}} \chi(zy)\chi(x(w-y))\e^{-\sigma t |w|^b}\e^{-\sigma t^\prime |y|^b}\,{\rm d}x\bigg)\,{\rm d}w\,{\rm d}y\notag%
\\&\qquad\qquad= \int_{\mathds Q_p}\int_{\mathds Q_p} \chi(zy)\bigg(\int_{{\Za}} \chi(x(w-y))\,{\rm d}x\bigg)\e^{-\sigma t |w|^b}\e^{-\sigma t^\prime |y|^b}\,{\rm d}w\,{\rm d}y.%
\end{align}
The equality \[\int_{{\Za}}\chi(xa)\,{\rm d}x = p^a{\mathds 1_{{\Zpa}}}(a)\] together with \eqref{sec3:intermediateineqtplustprime} implies that 
\begin{align*}
&\int_{ {\Zpa}}\rho(t,x)\rho(t^\prime,z-x)\,{\rm d}x \\&\hspace{1in}= \int_{\mathds Q_p}\int_{\mathds Q_p}\chi(zy){p^{a}\mathds 1}_{{\Zpa}}(w-y)\e^{-\sigma t |w|^b}\e^{-\sigma t^\prime |y|^b}\,{\rm d}w\,{\rm d}y%
\\&\hspace{1in} = p^a\int_{\mathds Q_p}\int_{{\Zpa}}\chi(zy)\e^{-\sigma t |y+w|^b}\e^{-\sigma t^\prime |y|^b}\,{\rm d}w\,{\rm d}y%
\\&\hspace{1in} = p^a\int_{\mathds Q_p}\chi(zy)\e^{-\sigma t^\prime |y|^b}\bigg(\int_{{\Zpa}}\e^{-\sigma t |y+w|^b}\,{\rm d}w\bigg)\,{\rm d}y%
\\&\hspace{1in} > \int_{\mathds Q_p}\chi(zy)\e^{-\sigma(t+t^\prime) |y|^b}\,{\rm d}y \cdot p^a\int_{{\Zpa}}\e^{-\sigma t |w|^b}\,{\rm d}w%
\\&\hspace{1in} =  \rho\big(t+t^\prime, z\big)  p^a\int_{{\Zpa}}\e^{-\sigma t |w|^b}\,{\rm d}w = \rho\big(t+t^\prime, z\big) \int_{{\Za}}\rho(t, x)\,{\rm d}x%
\end{align*}
where Lemma~\ref{6:lem:2} implies the inequality and Lemma~\ref{6:lem:1} implies the ultimate equality.

\end{proof}

\begin{corollary}\label{6:cor:3}
Suppose that $\rho$ is the fundamental solution to \eqref{eq1}.  For all positive real numbers $t$ and $t^\prime$ and all $z$ in ${\Za}$, \[\int_{{\Za}}\rho(t,x)\rho(t^\prime,z-x)\,{\rm d}x > \rho(t+t^\prime, z)\int_{{\Za}}\rho(t^\prime,x)\,{\rm d}x.\]
\end{corollary}

\begin{proof}
Use a change of variables to obtain the equality
\[\int_{{\Za}}\rho(t,x)\rho(t^\prime,z-x)\,{\rm d}x = \int_{{\Za}}\rho(t,z-u)\rho(t^\prime,u)\,{\rm d}u.\]  The result then follows immediately from Proposition~\ref{6:prop:3}.
\end{proof}

\begin{proposition}\label{6:prop:4}
Suppose that $m$ is a natural number. Suppose that $y$ is in $B_a$ and $n$ is a natural number large enough so that $B_{-n}(y)$ is a subset of $B_a$.  If for each $i$ in $\{1, \dots, m+1\}$ the real number $t_i$ is positive, then
\begin{align*}&\int_{{\Za}}\cdots\int_{{\Za}}\cdot\int_{B_{-n}(y)} \rho(t_1, z_1)\rho(t_2, z_2-z_1)\\&\hspace{.5in}\cdots \rho(t_m, z_m-z_{m-1})\rho(t_{m+1}, z-z_m)\,{\rm d}z\,{\rm d}z_m\,\cdots \,{\rm d}z_1 \\&\hspace{.9in}> \int_{B_{-n}(y)} \rho(t_1 + \cdots + t_m + t_{m+1}, z)\,{\rm d}z \cdot \prod_{1\leq i\leq m} \int_{{\Za}} \rho(t_i, z_i)\,{\rm d}z_i.\end{align*}
\end{proposition}

\begin{proof}
If $m$ is equal to 1, then the corollary to Proposition~\ref{6:prop:3} implies inequality
\begin{align}\label{6:lem:3:a}\int_{{\Za}}\rho(t_2,z_1)\rho(t_1,z-z_1)\,{\rm d}x > \rho(t_1 + t_2, z)\int_{{\Za}}\rho(t_2,z_1)\,{\rm d}z_1.\end{align} Integrate both sides of \eqref{6:lem:3:a} to obtain, \begin{align}\label{6:eqlem:3:a}&\int_{B_{-n}(y)}\int_{{\Za}}\rho(t_2,z_1)\rho(t_1,z-z_1)\,{\rm d}z_1\,{\rm d}z  \notag\\&\hspace{1in}> \int_{B_{-n}(y)} \rho(t_1 + t_2, z)\int_{{\Za}}\rho(t_1,z_1)\,{\rm d}z_1\,{\rm d}z\notag\\& \hspace{1in}= \int_{B_{-n}(y)} \rho(t_1 + t_2, z) \,{\rm d}z \int_{{\Za}}\rho(t_1,z_1)\,{\rm d}z_1.\end{align}  Switch the order of integration on the left hand side of \eqref{6:eqlem:3:a} to verify the proposition when $m$ is equal to 1.

Let $z_0$ be equal to 0 and let $\ell$ be an arbitrary natural number.  Suppose that for any epoch $(t_1, \dots, t_{\ell})$ of length $\ell$,
\begin{align*}&\int_{{\Za}}\cdots\int_{{\Za}}\cdot\int_{B_{-n}(y)} \rho(t_{\ell+1}, z-z_\ell)\prod_{i\leq 1\leq \ell} \rho(t_i, z_i-z_{i-1})\,{\rm d}z\,{\rm d}z_\ell\,\cdots \,{\rm d}z_1 \\&\hspace{1in}> \int_{B_{-n}(y)} \rho\bigg(t_{\ell+1} + \sum_{1\leq k\leq \ell } t_k, z\bigg)\,{\rm d}z \cdot \prod_{1\leq i\leq \ell} \int_{{\Za}} \rho(t_i, z_i)\,{\rm d}z_i.\end{align*}  Suppose that $(t_1, \dots, t_{\ell}, t_{\ell+1})$ is an epoch of length $\ell+1$.  Change variables by taking \[w_i =  z_i - z_{i-1}\] to obtain the equality
\begin{align*}
&\int_{{\Za}}\cdots\int_{{\Za}}\int_{{\Za}}\cdot\int_{B_{-n}(y)} \rho(t_{\ell+2}, z-z_{\ell+1})\\&\hspace{1.75in}\cdot\prod_{i\leq 1\leq \ell+1} \rho(t_i, z_i-z_{i-1})\,{\rm d}z\,{\rm d}z_{\ell+1}\,{\rm d}z_\ell\,\cdots \,{\rm d}z_1 
\\&\hspace{.5in} = \int_{{\Za}}\cdots\int_{{\Za}}\int_{{\Za}}\cdot\int_{B_{-n}(y)} \rho(t_{\ell+2}, z-(w_1 + \cdots + w_\ell+ w_{\ell+1}))\\&\hspace{1.75in}\cdot\prod_{i\leq 1\leq \ell+1} \rho(t_i, w_i)\,{\rm d}z\,{\rm d}w_{\ell+1}\,{\rm d}w_\ell\,\cdots \,{\rm d}w_1 = \ast_1.
\end{align*}
Rearrange terms and then use the corollary to Proposition~\ref{6:prop:3} to obtain the inequality
\begin{align*}
\ast_1 & = \int_{{\Za}}\cdots\int_{{\Za}} \Bigg(\int_{{\Za}}\cdot\int_{B_{-n}(y)} \rho(t_{\ell+2}, (z-w_1 - \cdots - w_\ell)- w_{\ell+1})\\&\hspace{1.4in}\cdot\rho(t_{\ell+1}, w_{\ell+1})\,{\rm d}z\,{\rm d}w_{\ell+1}\Bigg)\prod_{i\leq 1\leq \ell} \rho(t_i, w_i)\,{\rm d}w_\ell\,\cdots \,{\rm d}w_1 
\\& > \int_{{\Za}}\cdots\int_{{\Za}}\Bigg(\int_{B_{-n}(y)} \rho(t_{\ell+2}+t_{\ell+1}, (z-w_1 - \cdots - w_\ell)) \,{\rm d}z \\&\hspace{1in}\cdot\int_{{\Za}}\rho(t_{\ell+1},w_{l+1})\,{\rm d}w_{l+1}\Bigg)
\prod_{i\leq 1\leq \ell} \rho(t_i, w_i)\,{\rm d}w_\ell\,\cdots \,{\rm d}w_1 = \ast_2.
\end{align*}
Rearrange terms and then use the inductive hypothesis to obtain the inequality
\begin{align*}
\ast_2  &= \int_{{\Za}}\cdots\int_{{\Za}} \Bigg(\int_{B_{-n}(y)} \rho(t_{\ell+2} + t_{\ell+1}, z-z_\ell) \,{\rm d}z\\&\hspace{1.5in}\cdot \prod_{i\leq 1\leq \ell} \rho(t_i, z_i-z_{i-1})\,{\rm d}w_\ell \Bigg)\int_{{\Za}}\rho(t_{\ell+1},z_{l+1})\,{\rm d}z_{l+1}
\\& > \int_{B_{-n}(y)} \rho\bigg(t_{\ell+2} + t_{\ell+1} + \sum_{1\leq i\leq \ell} t_i, z\bigg)\,{\rm d}z\\& \hspace{1.5in}\cdot\prod_{1\leq i\leq \ell} \int_{{\Za}} \rho(t_i, z_i)\,{\rm d}z_i\int_{{\Za}}\rho(t_{\ell+1},z_{l+1})\,{\rm d}z_{l+1}%
\\& = \int_{B_{-n}(y)} \rho\bigg(\sum_{1\leq i\leq \ell+2} t_i, z\bigg)\,{\rm d}z \cdot \prod_{1\leq i\leq \ell+1} \int_{ {\Za}} \rho(t_i, z_i)\,{\rm d}z_i.
\end{align*}
Fubini's Theorem therefore implies that \begin{align*}&\int_{{\Za}}\cdots\int_{{\Za}}\cdot\int_{B_{-n}(y)} \rho(t_{\ell+2}, z-z_{\ell+1})\prod_{i\leq 1\leq \ell} \rho(t_i, z_i-z_{i-1})\,{\rm d}z\,{\rm d}z_\ell\,\cdots \,{\rm d}z_1 \\&\hspace{1in}> \int_{B_{-n}(y)} \rho\bigg(\sum_{1\leq k\leq {\ell+2} } t_k, z\bigg)\,{\rm d}z \cdot \prod_{1\leq i\leq {\ell+1}} \int_{{\Za}} \rho(t_i, z_i)\,{\rm d}z_i,\end{align*} and so the axiom of induction implies the proposition.
\end{proof}

\begin{corollary}\label{6:cor:4}
Suppose that $m$ is a natural number.   For all $y$ in $B_a$, if $n$ is a natural number large enough so that $B_{-n}(y)$ is a subset of $B_a$, then
\begin{align*}&P\Big(X_t\in B_{-n}(y)\;{\Big|}\;X_{\frac{t}{m}}\in {\Za}\cap \cdots \cap X_{\frac{(m-1)t}{m}}\in {\Za} \cap X_t\in {\Za}\Big) \\&\hspace{3.5in}> P\big(X_t\in B_{-n}(y)\big).\end{align*}
\end{corollary}

\begin{proof}

Ultrametricity of the valuation on $\mathds Q_p$ implies that \begin{align*}&P\Big(X_{\frac{t}{m}}\in {\Za}\cap \cdots \cap X_{t}\in {\Za}\Big) \\&\hspace{1in}= P\Big(X_{\frac{t}{m}} - X_0\in {\Za}\cap \cdots \cap X_{t} - X_{\frac{(m-1)t}{m}}\in {\Za}\Big).\end{align*} Therefore, 
\begin{align*}
&P\Big(X_t\in B_{-n}(y)\;{\Big|}\;X_{\frac{t}{m}}\in {\Za}\cap \cdots \cap X_{\frac{(m-1)t}{m}}\in {\Za}\Big)\\
& \hspace{1in}= \dfrac{P\Big(X_t\in B_{-n}(y)\cap X_{\frac{t}{m}}\in {\Za}\cap \cdots \cap X_{\frac{(m-1)t}{m}}\in {\Za}\Big)}{P\Big(X_{\frac{t}{m}}\in {\Za}\cap \cdots \cap X_{\frac{(m-1)t}{m}}\in {\Za}\cap X_{t}\in {\Za}\Big)}\\
& \hspace{1in}= \dfrac{P\Big(X_t\in B_{-n}(y)\cap X_{\frac{t}{m}}\in {\Za}\cap \cdots \cap X_{\frac{(m-1)t}{m}}\in {\Za}\Big)}{P\Big(X_{\frac{t}{m}}-X_0\in {\Za}\cap \cdots \cap X_{t}- X_{\frac{(m-1)t}{m}}\in {\Za}\Big)}\\
& \hspace{1in}= \dfrac{P\Big(X_t\in B_{-n}(y)\cap X_{\frac{t}{m}}\in {\Za} \cap \cdots \cap X_{\frac{(m-1)t}{m}}\in {\Za}\Big)}{P\Big(X_{\frac{t}{m}}-X_0\in {\Za}\Big) \cdots P\Big(X_{t}- X_{\frac{(m-1)t}{m}}\in {\Za}\Big)}
\end{align*}
where independence of the increments of the process implies the ultimate equality.  Write the above probabilities as integrals to obtain
\begin{align*}
&P\Big(X_t\in B_{-n}(y)\;{\Big|}\;X_{\frac{t}{m}}\in {\Za}\cap \cdots \cap X_{\frac{(m-1)t}{m}}\in {\Za} \cap X_t\in {\Za}\Big) \\&\hspace{.25in} = \int_{{\Za}}\cdots\int_{{\Za}}\cdot\int_{B_{-n}(y)} \rho\Big(\frac{t}{m}, z_1\Big)\rho\Big(\frac{t}{m}, z_2-z_1\Big)\cdots \rho\Big(\frac{t}{m}, z_{m-1}-z_{m-2}\Big) \\&\hspace{.92in}\cdot \rho\Big(\frac{t}{m}, z-z_{m-1}\Big)\,{\rm d}z\,{\rm d}z_{m-1}\,\cdots \,{\rm d}z_1\cdot \frac{1}{\prod_{1\leq i\leq m} \int_{{\Za}} \rho\big(\frac{t}{m}, z_i\big)\,{\rm d}z_i}\\&\hspace{.25in}> \int_{B_{-n}(y)} \rho\Big(\frac{t}{m} + \cdots + \frac{t}{m}, z\Big)\,{\rm d}z \cdot \prod_{1\leq i\leq m} \int_{{\Za}} \rho\Big(\frac{t}{m}, z_i\Big)\,{\rm d}z_i\\&\hspace{3.25in} \cdot \frac{1}{\prod_{1\leq i\leq m} \int_{{\Za}} \rho\big(\frac{t}{m}, z_i\big)\,{\rm d}z_i}\\&\hspace{.25in}= \int_{B_{-n}(y)} \rho(t, z)\,{\rm d}z = P\big(X_t\in B_{-n}(y)\big),
\end{align*}
where Proposition~\ref{6:prop:4} implies the inequality above.
\end{proof}

\begin{proposition}\label{6:prop:5}
For all $y$ in ${\Za}$, if $n$ is a positive integer that is large enough so that $B_{-n}(y)$ is a subset of $B_a$, then \[P\big(X_t\in B_{-n}(y){\big |}\;||X||_t\leq p^a\big) > P\big(X_t\in B_{-n}(y)\big).\]
\end{proposition}

\begin{proof}
Denote by $||X||_t^{(m)}$ the random variable given by\[||X||_t^{(m)} = \max\Big\{|X_i|\colon i\in \Big\{\frac{t}{m}, \frac{2t}{m}, \dots, \frac{(m-1)t}{m}, t\Big\}\Big\}.\]  Since \[||X||_t^{(m)} \leq p^a\] if and only if for all $i$ in $\{1, \dots, m\}$, \[X_{\frac{it}{m}}\in {\Za},\] the corollary to Proposition \ref{6:prop:4} implies that%
\begin{align*}
&P\big(X_T\in B_{-n}(y) {\big|}\;||X||_t^{(m)}\leq p^a\big)  
\\&\qquad >  P\big(X_t\in B_{-n}(y)\big)\dfrac{P\big(\big(X_t\in B_{-n}(y)\big)\cap \big(||X||_t^{(m)}\leq 1\big)\big)}{P\big(||X||_t^{(m)}\leq p^a\big)}%
\\& \qquad = \dfrac{P\Big(X_{\frac{t}{m}}\in {\Za}\cap \cdots \cap X_{\frac{(m-1)t}{m}}\in {\Za}\cap X_t\in {\Za} \cap X_t\in B_{-n}(y)\Big)}{P\big(||X||_t^{(m)}\leq p^a\big)}%
\\& \qquad = \dfrac{P\Big(X_{\frac{t}{m}}\in {\Za}\cap \cdots \cap X_{\frac{(m-1)t}{m}}\in {\Za}\cap X_t\in B_{-n}(y)\Big)}{P\big(||X||_t^{(m)}\leq p^a\big)}%
\\& \qquad = \dfrac{P\Big(X_t\in B_{-n}(y)\;{\big|}\;X_{\frac{t}{m}}\in {\Za}\cap \cdots \cap X_{\frac{(m-1)t}{m}}\in {\Za}\Big)}{P\big(||X||_t^{(m)}\leq p^a\big)}%
\\&\hspace{2.5in} \cdot P\Big(X_{\frac{t}{m}}\in {\Za}\cap \cdots \cap X_{\frac{(m-1)t}{m}}\in {\Za}\Big)%
\\& \qquad > P\big(X_t\in B_{-n}(y)\big)\dfrac{P\Big(X_{\frac{t}{m}}\in {\Za}\cap \cdots \cap X_{\frac{(m-1)t}{m}}\in {\Za}\Big)}{P\big(||X||_t^{(m)}\leq p^a\big)},%
\end{align*}
hence %
\begin{align}\label{LLimits:of:Conditioned}
&P\big(X_T\in B_{-n}(y) {\big|}\;||X||_t^{(m)}\leq p^a\big) \notag\\&\hspace{.5in}> P\big(X_t\in B_{-n}(y)\big)\dfrac{P\Big(X_{\frac{t}{m}}\in\Za\cap \cdots \cap X_{\frac{(m-1)t}{m}}\in\Za\Big)}{P\big(||X||_t^{(m)}\leq p^a\big)}.
\end{align}
 The right continuity of the Skorohod paths implies that \begin{align}\label{sec3:aaa}\lim_{m\to \infty} P\big(||X||_t^{(m)} \leq p^a\big) = P\big(||X||_t\leq p^a\big),\end{align} that \begin{align}\label{sec3:bbb}\lim_{m\to \infty} P\big(X_t \in B_{-n}(y)\cap ||X||_t^{(m)} \leq p^a\big) = P\big(X_t \in B_{-n}(y)\cap ||X||_t\leq p^a\big),\end{align} and that \begin{align}\label{sec3:ccc}\lim_{m\to \infty} P\Big(X_{\frac{t}{m}}\in {\Za}\cap \cdots \cap X_{\frac{(m-1)t}{m}}\in {\Za}\Big) = P\big(||X||_t^{(m)}\leq p^a\big).\end{align}   Use the fact that \[P\big(X_t\in B_{-n}(y)\cap ||X||_t^{(m)} \leq p^a\big)\ne 0\] and \[P\big(X_t\in B_{-n}(y)\cap ||X||_t\leq p^a\big) \ne 0\] together with the equalities \eqref{sec3:aaa}, \eqref{sec3:bbb}, and \eqref{sec3:ccc} to obtain the equalities \begin{align*}&\lim_{m\to \infty}P\big(X_t\in B_{-n}(y)\big)\dfrac{P\Big(X_{\frac{t}{m}}\in {\Za}\cap \cdots \cap X_{\frac{(m-1)t}{m}}\in {\Za}\Big)}{P\big(||X||_t^{(m)}\leq p^a\big)} = P\big(X_t\in B_{-n}(y)\big)\end{align*} and \[\lim_{m\to \infty} P\big(X_T\in B_{-n}(y) {\big|}\;||X||_t^{(m)}\leq p^a\big) = P\big(X_T\in B_{-n}(y) {\big|}\;||X||_t\leq p^a\big).\] These two equalities together with \eqref{LLimits:of:Conditioned} imply the proposition.
\end{proof}

\begin{proposition}\label{6:prop:5b}
For all $y$ in ${\Za}$, if $n$ is a positive integer that is large enough so that $B_{-n}(y)$ is a subset of $B_a$, then \[P\big(||X||_t\leq p^a{\big |}\;X_t\in B_{-n}(y)\big) > P\big(||X||_t\leq p^a\big).\]
\end{proposition}

\begin{proof}

Proposition~\ref{6:prop:5} implies that
\begin{align*}
&P\big(||X||_t\leq p^a{\big |}\;X_t\in B_{-n}(y)\big) \\&\hspace{.5in}= \dfrac{P\big(||X||_t\leq p^a\cap X_t\in B_{-n}(y)\big)}{P\big(X_t\in B_{-n}(y)\big)}\\%
&\hspace{.5in}= \dfrac{P\big(X_t\in B_{-n}(y) \cap ||X||_t\leq p^a\big)}{P\big(||X||_t\leq p^a\big)}\dfrac{P\big(||X||_t\leq p^a\big)}{P\big(X_t\in B_{-n}(y)\big)}\\%
&\hspace{.5in}= P\big(X_t\in B_{-n}(y){\big |}\; ||X||_t\leq p^a\big)\dfrac{P\big(||X||_t\leq p^a\big)}{P\big(X_t\in B_{-n}(y)\big)}\\%
&\hspace{.5in}> P\big(X_t\in B_{-n}(y)\big)\dfrac{P\big(||X||_t\leq p^a\big)}{P\big(X_t\in B_{-n}(y)\big)} = P\big(||X||_t\leq p^a\big).%
\end{align*}

\end{proof}

\begin{theorem}\label{6:thm:6}
Suppose that $x$ is in $\mathds Q_p$ and $y$ is in ${\Za}+x$.  For all $t$ in $(0,T]$, \[P_{t, x, y}\big(||X-x||_T\leq p^a\big) \ge P_x(||X-x||_T\leq p^a).\]
\end{theorem}

\begin{proof}
Denote by $z$ the $p$-adic number $y-x$.  Write the measure on the brownian bridges as a limit of conditional measures with respect to events of nonzero probability to obtain
\begin{align}\label{sec4:eq:in:pf}
P\big(||X||_t\leq p^a\;{\big |}\;X_t = z\big) &= \lim_{n\to\infty}P\big(||X||_t\leq p^a{\big |}\;X_t\in B_{-n}(y)\big)%
\notag\\& \ge \lim_{n\to \infty}P\big(||X||_t\leq p^a\big) = P\big(||X||_t\leq p^a\big),
\end{align}
where Proposition~\ref{6:prop:5b} implies the inequality \eqref{sec4:eq:in:pf}.  Independence of the increments of the process implies that 
\begin{align}\label{sec4:eq:in:pf:a}
P\big(||X||_T\leq p^a\;{\big |}\;X_t = z\big) &= P\big(||X||_t\leq p^a\;{\big |}\;X_t = z\big)P_z\big(||X||_{T-t}\leq p^a\big)\notag\\&\ge P\big(||X||_t\leq p^a\big)P_z\big(||X||_{T-t}\leq p^a\big)
\\&= P\big(||X||_t\leq p^a\big)P\big(||X||_{T-t}\leq p^a\big)\notag\\& = P\big(||X||_T\leq p^a\big),\notag
\end{align}
where the inequality \eqref{sec4:eq:in:pf:a} follows from \eqref{sec4:eq:in:pf}. The inequality 
\begin{equation}\label{sec4:eq:in:pf:b}
P\big(||X||_T\leq p^a\;{\big |}\;X_t = z\big) \ge P\big(||X||_T\leq p^a\big)
\end{equation}
together with the equalities 
\begin{equation*}\label{sec4:end:1}
P\big(||X||_T\leq p^a\;{\big |}\;X_t = z\big)  = P_{t, 0, y-x}\big(||X||_T\leq p^a\big) =   P_{t, x, y}\big(||X||_T\leq p^a\big) 
\end{equation*}
and
\begin{equation*}\label{sec4:end:2}
P(||X||_T\leq p^a) = P_x(||X-x||_T\leq p^a),
\end{equation*} 
imply that \[P_{t, x, y}\big(||X-x||_T\leq p^a\big) \ge P_x(||X-x||_T\leq p^a).\]
\end{proof}



%
%
\end{document}